\documentclass{article}
\usepackage{verbatim,amsmath,hyperref,amssymb}
\usepackage{amsthm}
\usepackage{graphics,wrapfig,graphicx}
\usepackage{palatino}
\usepackage{color}
\usepackage{tikz}
\usepackage{multirow}
\usepackage{textcomp}

\definecolor{darkblue}{rgb}{0,0,.5}
\definecolor{redblue}{rgb}{.5,0,.2}
\definecolor{greenblue}{rgb}{0,.4,.3}
\definecolor{purpleblue}{rgb}{.4,0,.5}
\definecolor{darkgreen}{rgb}{0,0.4,0}
\definecolor{NavyBlue}{cmyk}{0.94,0.54,0,0}
\definecolor{JungleGreen}{cmyk}{0.99,0,0.52,0}
\definecolor{lightgray}{rgb}{0.9,0.9,0.9}

\newcounter{enucount}

\newtheorem{definition}[enucount]{Definition}
\newtheorem{theorem}[enucount]{Theorem}
\newtheorem{lemma}[enucount]{Lemma}
\newtheorem{proposition}[enucount]{Proposition}
\newtheorem{corollary}[enucount]{Corollary}

\let\emptyset\varnothing

\newcommand{\ZZ}{\mathbb{Z}}
\newcommand{\RR}{\mathbb{R}}

\newcommand{\FF}{\textup{faces}}
\newcommand{\FT}{\textup{facets}}
\newcommand{\VV}{\textup{vert}}
\newcommand{\xc}{\textup{xc}}
\newcommand{\conv}{\textup{conv}}
\newcommand{\forb}{\textup{forb}}
\newcommand{\rem}{X}

\newcommand{\cut}{\textup{CUT}}
\newcommand{\tsp}{\textup{TSP}}
\newcommand{\sub}{\textup{SUB}}
\newcommand{\unit}{\textbf{1}}

\newcommand{\st}{\textup{s.t.}}

\newcommand{\gustavo}[1]{\textcolor{black}{#1}}

\title{Forbidden vertices}
\author{Gustavo Angulo\textsuperscript{*}\footnote{Georgia Institute of Technology and Pontificia Universidad Cat\'olica de Chile: gangulo@gatech.edu}, 
Shabbir Ahmed\textsuperscript{\dag}, Santanu S. Dey\textsuperscript{\dag}\footnote{Georgia Institute of Technology: sahmed@isye.gatech.edu, santanu.dey@isye.gatech.edu}, 
Volker Kaibel\textsuperscript{\ddag}\footnote{Otto-von-Guericke-Universit\"{a}t Magdeburg: kaibel@ovgu.de}}

\begin{document}
\maketitle

\begin{abstract}
In this work, we introduce and study the forbidden-vertices problem. Given a polytope $P$ and a subset $\rem$ of its vertices, we study the complexity of linear optimization over the subset of vertices of $P$ that are not contained in $\rem$. This problem is closely related to finding the $k$-best basic solutions to a linear problem. We show that the complexity of the problem changes significantly depending on the encoding of both $P$ and $\rem$. We provide additional tractability results and extended formulations when $P$ has binary vertices only. Some applications and extensions to integral polytopes are discussed.
\end{abstract}

\section{Introduction}

Given a nonempty rational polytope $P\subseteq\RR^n$, we denote by $\VV(P)$, $\FF(P)$, and $\FT(P)$ the sets of vertices, faces, and facets of $P$, respectively, and we write $f(P):=|\FT(P)|$. We also denote by $\xc(P)$ the extension complexity of $P$, that is, the minimum number of inequalities in any linear extended formulation of $P$, i.e., a description of a polyhedron whose image under a linear map is $P$ \gustavo{(see for instance \cite{exponentialFiorini}.)} Finally, given a set $\rem \subseteq \VV(P)$, we define $\forb(P,\rem):=\conv(\VV(P)\setminus \rem)$, where $\conv(S)$ denotes the convex hull of $S\subseteq\RR^n$. This work is devoted to understanding the complexity of the forbidden-vertices problem defined below.

\begin{definition}\label{def_forbidden}
 Given a polytope $P\subseteq \RR^n$, a set $\rem\subseteq \VV(P)$, and a vector $c\in\RR^n$, the forbidden-vertices problem is to either assert $\VV(P)\setminus \rem=\emptyset$, or to return a minimizer of $c^\top x$ over $\VV(P)\setminus \rem$ otherwise.
\end{definition}

Our work is motivated by enumerative schemes for stochastic integer programs \cite{Laporte}, where a series of potential solutions are evaluated and discarded from the search space. As we will see later, the problem is also related to finding different basic solutions to a linear program.

To address the complexity of the forbidden-vertices problem, it is crucial to distinguish between different encodings of a polytope.

\begin{definition}
An explicit description of a polytope $P\subseteq \RR^n$ is a system $Ax\leq b$ defining $P$. An implicit description of $P$ is a separation oracle which, given a rational vector $x\in\RR^n$, either asserts $x\in P$, or returns a valid inequality for $P$ that is violated by $x$.
\end{definition}

Note that an extended formulation for $P$ is a particular case of an implicit description. When $P$ admits a separation oracle that runs in time bounded polynomially in the facet complexity of $P$ and the encoding size of the point to separate, we say that $P$ is tractable. We refer the reader to \cite[Section 14]{schrijver1998theory} for a deeper treatment of the complexity of linear programming.

We also distinguish different \gustavo{encodings} of a set of vertices.

\begin{definition}
An explicit description of $\rem\subseteq\VV(P)$ is the list of the elements in $\rem$. If $\rem=\VV(F)$ for some face $F$ of $P$, then an implicit description of $\rem$ is an encoding of $P$ and some valid inequality for $P$ defining $F$.
\end{definition}

Below we summarize our main contributions.

\begin{itemize}
 \item In Section~\ref{general}, we show that the complexity of optimizing over $\VV(P)\setminus \rem$ or describing $\forb(P,\rem)$ changes significantly depending on the encoding of $P$ and/or $\rem$. In most situations, however, the problem is hard.
 \item In Section~\ref{binary} we consider the case of removing a list $\rem$ of binary vectors from a 0-1 polytope $P$. When $P$ is the unit \gustavo{cube}, we present two compact extended formulations describing $\forb([0,1]^n,\rem)$. We further extend this result and show that the forbidden-vertices problem is polynomially solvable for tractable 0-1 polytopes.
 \item Then in Section~\ref{app} we apply our results to the $k$-best problem and to binary all-different polytopes, showing the tractability of both. Finally, in Section~\ref{integral}, we also provide extensions to integral polytopes.
\end{itemize}

The complexity results of Sections~\ref{general} and \ref{binary} lead to the \gustavo{classification shown in Table~\ref{classif}}, depending on the \gustavo{encoding} of $P$ and $\rem$, and whether $P$ has 0-1 vertices only or not. \gustavo{Note that ($*$) is implied, for instance, by Theorem~\ref{binary_facet}. Although we were not able to establish the complexity of ($**$), Proposition~\ref{faceTU} presents a tractable subclass.}

\begin{table}[ht]
\footnotesize
\begin{center}
\begin{tabular}{cc|cc|cc|}
~ & ~ & \multicolumn{4}{c|}{$P$}\\
~ & ~ & \multicolumn{2}{c|}{General}&\multicolumn{2}{c|}{0-1}\\
~ & ~ & Explicit & Implicit & Explicit & Implicit\\
\hline
\multirow{4}{*}{$\rem$} & \multirow{2}{*}{Explicit} & $\mathcal{NP}$-hard \gustavo{(Thm.~\ref{hardness})} & \multirow{2}{*}{$\mathcal{NP}$-hard for $|\rem|=1$ \gustavo{(Thm.~\ref{cutpoly})}}  & \multirow{2}{*}{Polynomial} & \multirow{2}{*}{Polynomial \gustavo{(Thm.~\ref{poly01})}}\\
 &  & Polynomial for fixed $|\rem|$ \gustavo{(Prop.~\ref{polyFixed})}&   &  & \\
 &&&&&\\
 & Implicit & $\mathcal{NP}$-hard \gustavo{(Prop.~\ref{knapsack_facet})}& $\mathcal{NP}$-hard \gustavo{($*$)}& \gustavo{($**$)} & $\mathcal{NP}$-hard \gustavo{(Thm.~\ref{binary_facet})}
 \label{classif}
\end{tabular}
\caption{Complexity classification.}
\end{center}
\end{table}

In constructing linear extended formulations, disjunctive programming emerges as a practical powerful tool. The lemma below follows directly from \cite{Balas} and the definition of extension complexity. We will frequently refer to it.

\begin{lemma}\label{disj}
Let $P_1,\ldots,P_k$ be \gustavo{nonempty} polytopes in $\RR^n$. If $P_i=\{x\in\RR^n|\ \exists y_i\in\RR^{m_i}:\ E_ix+F_iy_i=h_i,\ y_i\geq 0\}$, then $\conv(\cup_{i=1}^k P_i)=\{x\in\RR^n|\ \exists x_i\in\RR^n,\ y_i\in\RR^{m_i},\ \lambda\in\RR^k:\ x=\sum_{i=1}^k x_i,\ E_ix_i+F_iy_i=\lambda_ih_i,\ \sum_{i=1}^k\lambda_i = 1,\ y_i\geq 0,\ \lambda\geq 0\}$. In particular, we have $\xc\left(\conv(\cup_{i=1}^k P_i)\right)\leq \sum_{i=1}^k (\xc(P_i)+1)$.
\end{lemma}

\section{General polytopes}\label{general}

We begin with some general results when $P\subseteq\RR^n$ is an arbitrary polytope. The first question is how complicated $\forb(P,\rem)$ is with respect to $P$.

\begin{proposition}\label{single}
 For each $n$, there exists a polytope $P_n\subseteq\RR^n$ and a vertex $v_n\in\VV(P_n)$ such that $P_n$ has $2n+1$ vertices and $n^2+1$ facets, while $\forb(P_n,\{v_n\})$ has $2^n$ facets.
\end{proposition}
\begin{proof}
 Let $Q_n:=[0,1]^n\cap L$, where $L:=\left\{x\in\RR^n|\ \textbf{1}^\top x\leq \frac{3}{2}\right\}$ and $\textbf{1}$ is the vector of ones. It has been observed \cite{Avis} that $Q_n$ has $2n+1$ facets and $n^2+1$ vertices. We translate $Q_n$ and define $Q_n':=Q_n- \frac{1}{n}\textbf{1}=\left[-\frac{1}{n},1-\frac{1}{n}\right]^n\cap L'$, where $L':=\left\{x\in\RR^n|\ \textbf{1}^\top x\leq \frac{1}{2}\right\}$. Since $Q_n'$ is a full-dimensional polytope having the origin in its interior, there is a one-to-one correspondence between the facets of $Q_n'$ and the vertices of its polar $P_n:=(Q_n')^*$ and vice versa. In particular, $P_n$ has $n^2+1$ facets and $2n+1$ vertices. Let $v\in\VV(P_n)$ be the vertex associated with the facet of $Q_n'$ defined by $L'$. From polarity, we have $\forb(P_n,\{v\})^*=\left[-\frac{1}{n},1-\frac{1}{n}\right]^n$. Thus $\forb(P_n,\{v\})^*$ is a full-dimensional polytope with the origin in its interior and $2^n$ vertices. By polarity, we obtain that $\forb(P_n,\{v\})$ has $2^n$ 
facets.
\end{proof}

Note that the above result only states that $\forb(P,\rem)$ may need exponentially many inequalities to be described, which does not constitute a proof of hardness. Such a result is provided by Theorem~\ref{hardness} at the end of this section. We first show that $\forb(P,\rem)$ has an extended formulation of polynomial size in $f(P)$ when both $P$ and $\rem$ are given explicitly and the cardinality of $\rem$ is fixed.

\begin{proposition}\label{polyFixed}
Suppose $P = \{ x \in \RR^n|\ Ax \leq b\}$. Using this description of $P$, and an explicit list of vertices $\rem$, we can construct an extended formulation of $\forb(P,\rem) $ that requires at most $f(P)^{|\rem| + 1}$ inequalities, i.e., $\xc(\forb(P,\rem)) \leq f(P)^{|\rem| + 1}$.
\end{proposition}

\begin{proof}
Let $\rem=\{v_1,\ldots,v_{|\rem|}\}$ and define $\mathcal F_\rem:=\{F_1\cap\cdots\cap F_{|\rem|}|\ F_i\in \FT(P),\ v_i\notin F_i,\ i=1,\ldots,|\rem|\}$. We claim
$$\forb(P,\rem)=\conv\left(\cup_{F\in\mathcal F_\rem}F\right).$$

Indeed, let $w\in \VV(P)\setminus \rem$. For each $i=1,\ldots,|\rem|$, there exists $F_i\in\FT(P)$ such that $w\in F_i$ and $v_i\notin F_i$. Therefore, letting $F:=F_1\cap\cdots\cap F_{|\rem|}$, we have $F\in\mathcal F_\rem$ and $w\in F$, proving the forward inclusion. For the reverse inclusion, consider $F\in\mathcal F_\rem$. By definition, $F$ is a face of $P$ that does not intersect $\rem$, and hence $F\subseteq \forb(P,\rem)$.

By Lemma~\ref{disj}, we have $\xc(\forb(P,\rem))\leq\sum_{F\in\mathcal F_\rem}(\xc(F)+1)$. Since $\xc(F)\leq f(F)\leq f(P)-1$ for each proper face $F$ of $P$ and $|\mathcal F_\rem|\leq f(P)^{|\rem|}$, the result follows.
\end{proof}

Note that when $\rem=\{v\}$, the above result reduces $\forb(P,\{v\})$ to the convex hull of the union of the facets of $P$ that are not incident to $v$, which is a more intuitive result. Actually, we can expect describing $\forb(P,\rem)$ to be easier when the vertices in $\rem$ are ``far'' thus can be removed ``independently'', and more complicated when they are ``close''. Proposition~\ref{polyFixed} can be refined as follows.

The graph of a polytope $P$, or the 1-skeleton of $P$, is a graph $G$ with vertex set $\VV(P)$ such that two vertices are adjacent in $G$ if and only if they are adjacent in $P$.

\begin{proposition}\label{components}
Let $G$ be the graph of $P$. Let $\rem\subseteq\VV(P)$ and let $(\rem_1,\ldots,\rem_m)$ be a partition of $\rem$ such that $\rem_i$ and $\rem_j$ are independent in $G$, i.e., there is no edge connecting $\rem_i$ to $\rem_j$, for all $1\leq i<j\leq m$. Then
$$\forb(P,\rem)=\bigcap_{i=1}^m\forb(P,\rem_i).$$
\end{proposition}
\begin{proof}
We only need to show $\forb(P,\rem)\supseteq\bigcap_{i=1}^m\forb(P,\rem_i)$. For this, it is enough to show that for each $c$ we have $\max\{c^\top x:\ x\in \forb(P,\rem)\}\geq \max\left\{c^\top x:\ x\in \bigcap_{i=1}^m\forb(P,\rem_i)\right\}$. Given $c$, let $v$ be an optimal solution to the maximization problem in the right-hand side, and let $W\subseteq\VV(P)$ be the set of vertices $w$ of $P$ such that $c^\top w\geq c^\top v$. \gustavo{Observe that $W$ induces a connected subgraph of the graph $G$ of $P$ since the simplex method applied to $\max\{c^\top x:\ x\in P\}$ starting from a vertex in $W$ visits elements in $W$ only.} Hence, due to the \gustavo{independence} of $\rem_1,\ldots,\rem_m$, either there is some $w\in W$ with $w\notin \rem_1\cup\cdots\cup\rem_m$, in which case we have $w\in \forb(P,\rem)$ and $c^\top w\geq c^\top v$ as desired, or $W\subseteq \rem_i$ for some $i$, which yields the contradiction $v\in\forb(P,\rem_i)\subseteq\forb(P,W)$ with $c^\top x<c^\top v$ for all $x\in\VV(P)\setminus W$.
\end{proof}

Conversely, we may be tempted to argue that if $\forb(P,\rem)=\forb(P,\rem_1)\cap\forb(P,\rem_2)$, then $\rem_1$ and $\rem_2$ are ``far''. However, this is not true in general. For instance, consider $P$ being a simplex. Then any $\rem\subseteq \VV(P)$ is a clique in the graph of $P$, and yet $\forb(P,\rem)=\forb(P,\rem_1)\cap\forb(P,\rem_2)$ for any partition $(\rem_1,\rem_2)$ of $\rem$.

Proposition~\ref{components} generalizes the main result of \cite{lee2003cropped} regarding cropped cubes. Moreover, the definition of being ``croppable'' in \cite{lee2003cropped} in the case of the unit cube coincides with the independence property of Proposition~\ref{components}.

Recall that a vertex of an $n$-dimensional polytope is simple if it is contained in exactly $n$ facets. \gustavo{Proposition~\ref{components} also implies the following well-known fact.}

\begin{corollary}\label{stable}
 \gustavo{If $\rem$ is independent in the graph of $P$ and all its elements are simple, then}
$$\forb(P,\rem)=P\cap\bigcap_{v\in \rem}H_v,$$

where $H_v$ is the half-space defined by the $n$ neighbors of $v$ that does not contain $v$.
\end{corollary}
\begin{proof}
The result follows from Proposition~\ref{components} since, as $\rem$ is simple, we have $\forb(P,\{v\})=P\cap H_v$ for any $v\in\rem$.
\end{proof}

Observe that when $P$ is given by an extended formulation or a separation oracle, $f(P)$ may be exponentially large with respect to the size of the encoding, and the bound given in Proposition~\ref{polyFixed} is not interesting. In fact, in this setting and using recent results on the extension complexity of the cut polytope \cite{Fiorini}, we show that removing a single vertex can render an easy problem hard. 

Let $K_n=(V_n,E_n)$ denote the complete graph on $n$ nodes. We denote by $\cut(n)$, $\cut^0(n)$, and $st\textrm-\cut(n)$ the convex hull of the characteristic vectors of cuts, nonempty cuts, and $st$-cuts of $K_n$, respectively.

\begin{theorem}\label{cutpoly}
 For each $n$, there exists a set $S_n\subseteq\RR^{n(n-1)/2}$ with $|S_n|=2^{n-1}+n-1$ and a point $v_n\in S_n$ such that linear optimization over $S_n$ can be done in polynomial time and $\xc(\conv(S_n))$ is polynomially bounded, but linear optimization over $S_n\setminus\{v_n\}$ is $\mathcal{NP}$-hard and $\xc(\conv(S_n\setminus\{v_n\}))$ grows exponentially.
\end{theorem}
\begin{proof}
Let $T_n:=\left\{n^2 \unit_e|\ e\in E_n\right\}$, where $\unit_e$ is the $e$-th unit vector, and define $S_n:=\VV\left(\cut^0(n)\right)\cup T_n$.

We have that linear optimization over $S_n$ can be done in polynomial time. To see this, suppose we are minimizing $c^\top x$ over $S_n$. Let $x^T$ and $x^C$ be the best solution in $T_n$ and $\cut^0(n)$, respectively. Note that computing $x^T$ is trivial, and if $c$ has a negative component, then $x^T$ is optimal. Otherwise, $c$ is nonnegative and $x^C$ can be found with a max-flow/min-cut algorithm. Then the best solution among $x^T$ and $x^C$ is optimal. Now, consider the dominant of $\cut^0(n)$ defined as $\cut^0(n)_+:=\cut^0(n)+\RR^{n(n-1)/2}_+$. From \cite{Conforti}, we have that $\cut^0(n)_+$ is an unbounded polyhedron having the same vertices as $\cut^0(n)$, and moreover, it has an extended formulation of polynomial size in $n$. Let $L:=\{x\in\RR^{n(n-1)/2}|\ \sum_{e\in E_n}x_e\leq n^2\}$. Then $\cut^0(n)_+\cap L$ is a polytope having two classes of vertices: those corresponding to $\VV\left(\cut^0(n)\right)$ and those belonging to the hyperplane defining $L$. Let $W$ be the latter set. Since $\conv(
W)\subseteq\conv(T_n)$, we obtain $\conv(S_n)=\conv\left(\cut^0(n)\cup T_n\right)=\conv\left((\cut^0(n)\cup W)\cup T_n)\right)=\conv\left((\cut^0(n)_+\cap L)\cup T_n\right)$. Applying disjunctive programming in the last expression yields a compact extended formulation for $\conv(S_n)$.

Now, let $v_n$ be any point from $T_n$, say the one corresponding to $\{s,t\}\in E$. We claim that linear optimization over $S_n\setminus \gustavo\{v_n\gustavo\}$ is $\mathcal{NP}$-hard. To prove this, consider an instance of $\max\{c^\top x|\ x \in st\textrm-\cut(n)\}$, where $c$ is a positive vector. Let $\bar c:= \max\{c_e|\ e \in E\}$. Let $d$ be obtained from $c$ as
$$d_e=\left\{\begin{array}{cc}
 c_e & e\neq\{s,t\}\\
 c_e+\bar c n^2 & e=\{s,t\}
\end{array}\right.$$

and consider the problem $\max\{d^\top x|\ x \in S_n\setminus \{v_n\}\}$. We have that every optimal solution to this problem must satisfy $x_{st} = 1$. Indeed, if $x \in T_n\setminus\{v_n\}$, then for some $e \in E_n\setminus\{\{s,t\}\}$ we have $d^\top x = d_e x_e = c_e n^2$. If $x \in \gustavo{\VV(\cut^0(n))}$ is not an $st$-cut, then $x_{st} = 0$ and thus $d^\top x \leq \bar c n^2$. On the other hand, if $x$ is an $st$-cut, then $x_{st} = 1$ and thus $d^\top x \geq d_{st} x_{st} = c_{st} + \bar c n^2$. Therefore $x_{st} = 1$ in any optimal solution, and in particular, such a solution must define an $st$-cut of maximum weight. Finally, since $x_{st}\leq 1$ defines a face of $\conv(S_n\setminus\{v_n\})$ and $\conv(S_n\setminus\{v_n\})\cap\{x\in\RR^{n(n-1)/2}|\ x_{st}=1\}=\gustavo{st\textrm{-}\cut(n)}$, we conclude that $\xc(\conv(S_n\setminus\{v_n\}))$ is exponential in $n$, \gustavo{for otherwise applying disjunctive programming over all pairs of nodes $s$ and $t$ would yield an extended formulation for $\cut(n)$ of polynomial size, contradicting the results in \cite{Fiorini}.}
\end{proof}

Contrasting Proposition~\ref{polyFixed} and Theorem~\ref{cutpoly} shows that the complexity of $\forb(P,\rem)$ depends on the encoding of $P$. On the other hand, in all cases analyzed so far, $\rem$ has been explicitly given as a list. Now we consider the case where $\rem=\VV(F)$ for some face $F$ of $P$.

\begin{proposition}\label{knapsack_facet}
\gustavo{Given a polytope $P\subseteq\RR^n$ and a face $F$, both described in terms of the linear inequalities defining them, optimizing a linear function over $\VV(P)\setminus\VV(F)$ is $\mathcal{NP}$-hard. Moreover,  $\xc(\conv(\VV(P)\setminus\VV(F)))$ cannot be polynomially bounded in the encoding length of the inequality description of $P$ and thus not in $n$.}
\end{proposition}
\begin{proof}
Let $a\in\ZZ^n_+$ and $b\in\ZZ_+$, and consider \gustavo{the binary knapsack set} $S:=\{x\in\{0,1\}^n|\ a^\top x\leq b\}$. Let $P:=\{x\in[0,1]^n|\ 2a^\top x\leq 2b+1\}$ and note that $S=P\cap\ZZ^n$. It is straightforward to verify that $x\in\VV(P)$ is fractional if and only if $2a^\top x=2b+1$. Then, if $F$ is the facet of $P$ defined by the previous constraint, we have $S=\VV(P)\setminus\VV(F)$. \gustavo{The second part of the statement is a direct consequence of \cite{pokutta2013note} using multipliers $4^i$ as discussed after Remark 3.4 of that reference.}
\end{proof}

It follows from Theorem~\ref{cutpoly} and Proposition~\ref{knapsack_facet} that only when $P$ and $\rem$ are explicitly given there is hope for efficient optimization over $\forb(P,\rem)$.  

In a similar vein, when the linear description of $P$ is provided, we can consider the vertex-enumeration problem, which consists of listing all the vertices of $P$. We say that such a problem is solvable in polynomial time if there exists an algorithm that returns the list in time bounded by a polynomial of $n$, $f(P)$, and the output size $|\VV(P)|$. In \cite{Khachiyan} it is shown that given a partial list of vertices, the decision problem ``is there another vertex?'' is $\mathcal{NP}$-hard for (unbounded) polyhedra, and  in \cite{Boros} this result is strengthened to polyhedra having 0-1 vertices only. Building on these results, we show hardness of the forbidden-vertices problem \gustavo{(Def.~\ref{def_forbidden})} for general polytopes.

\begin{theorem}\label{hardness}
The forbidden-vertices problem is $\mathcal{NP}$-hard, even if both $P$ and $\rem$ are explicitly given.
\end{theorem}
\begin{proof}
Let $Q=\{x\in\RR^n:\ Ax=b,\ x\geq 0\}$ be an unbounded polyhedron such that $\VV(Q)\subseteq\{0,1\}^n$. In \cite{Boros}, it is shown that given the linear description of $Q$ and a list $\rem\subseteq\VV(Q)$, it is $\mathcal{NP}$-hard to decide whether $\rem\neq\VV(Q)$. Let $P$ be the polytope obtained by \gustavo{intersecting} $Q$ with the half-space defined by $\sum_{i=1}^n x_i\leq n+1$, and let $F$ be the facet of $P$ associated with this constraint. Then we have $\VV(P)=\VV(Q)\cup\VV(F)$, $\sum_{i=1}^n x_i\leq n$ for $x\in\VV(Q)$, and $\sum_{i=1}^n x_i=n+1$ for $x\in\VV(F)$. Now, given the description of $P$ and a list $\rem\subseteq\VV(Q)\subseteq\VV(P)$, consider the instance of the forbidden-vertices problem $\min\left\{\sum_{i=1}^n x_i:\ x\in\VV(P)\setminus \rem\right\}$. The optimal value is equal to $n+1$ if and only if $\rem=\VV(Q)$. Since the reduction is clearly polynomial, the result follows.
\end{proof}

In fact, it also follows from \cite{Boros} that the forbidden-vertices problem for general polytopes becomes hard already for $|\rem|=n$.
Fortunately, the case of 0-1 polytopes is amenable to good characterizations.

\section{0-1 polytopes}\label{binary}

We consider polytopes having binary vertices only. We show that $\forb(P,\rem)$ is tractable as long as $P$ is \gustavo{and $\rem$ is explicitly given}. Our results for $P=[0,1]^n$ \gustavo{allow us} to obtain tractability in the case of general 0-1 polytopes.

\subsection{The 0-1 \gustavo{cube}}

In this subsection we have $P=[0,1]^n$, and therefore $\VV(P)=\{0,1\}^n$. We show the following result.

\begin{theorem}\label{P01minusV}
Let $\rem$ be a list of $n$-dimensional binary vectors. Then $\xc(\forb([0,1]^n,\rem)) \leq \mathcal O(n|\rem|)$.
\end{theorem}

For this, we present two extended formulations involving $\mathcal O(n|\rem|)$ variables and constraints. The first one is based on an identification between nonnegative integers and binary vectors.
The second one is built by recursion and lays ground for a simple combinatorial algorithm to optimize over $\forb([0,1]^n,\rem)$ and for an extension to remove vertices from general 0-1 polytopes.

\subsubsection{First extended formulation}

Let $N:=\{1,\ldots,n\}$ and $\mathcal N:=\{0,\ldots,2^n-1\}$. There exists a bijection between $\{0,1\}^n$ and $\mathcal N$ given by the mapping $\sigma(v):=\sum_{i\in N}2^{i-1}v_i$ for all $v\in \{0,1\}^n$. Therefore, we can write $\{0,1\}^n=\{v^0,\ldots,v^{2^n-1}\}$, where $v^k$ gives the binary expansion of $k$ for each $k\in \mathcal N$, that is, $v^k=\sigma^{-1}(k)$. Let $\rem=\{v^{k_1},\ldots,v^{k_m}\}$, where without loss of generality we assume $k_l<k_{l+1}$ for all $l=1,\ldots,m-1$. Also, let $\mathcal N_\rem:=\{k\in \mathcal N|\ v^k\in \rem\}$. Then we have
$$\{0,1\}^n\setminus \rem=\left\{x\in \{0,1\}^n|\ \sum_{i\in N} 2^{i-1}x_i \notin \mathcal N_\rem\right\}.$$

Now, for integers $a$ and $b$, let
$$K(a,b)=\left\{x\in \{0,1\}^n|\ a\leq \sum_{i\in N} 2^{i-1}x_i \leq b\right\}.$$

If $b<a$, then $K(a,b)$ is empty. Set $k_0=-1$ and $k_{m+1}=2^n$. Then we can write
$$\{0,1\}^n\setminus \rem=\bigcup_{l=0}^{m}K(k_l+1,k_{l+1}-1).$$

Thus
\begin{equation}
\forb([0,1]^n,\rem)=\conv\left(\bigcup_{l=0}^{m}K(k_l+1,k_{l+1}-1)\right)=\conv\left(\bigcup_{l=0}^{m}\conv(K(k_l+1,k_{l+1}-1))\right). \label{intervals}
\end{equation}

For $k\in\mathcal N$, let $N^k:=\{i\in N|\ v^k_i=1\}$. From \cite{Muldoon} we have

$$\conv(K(a,b))=\left\{x\in[0,1]^n:\ 
\begin{array}{rl}
\displaystyle \sum_{j\notin N^a|\ j>i}x_j\geq 1 -x_i& \forall i\in N^a\\
\displaystyle \sum_{j\in N^b|\ j>i}(1-x_j)\geq x_i & \forall i\notin N^b
\end{array}\right\},$$

thus $\conv(K(a,b))$ has $\mathcal O(n)$ facets. Finally, combining this and (\ref{intervals}), by Lemma~\ref{disj}, we have that $\forb([0,1]^n,\rem)$ can be described by an extended formulation having $\mathcal O(n|\rem|)$ variables and constraints.

\subsubsection{Second extended formulation}

Given $\rem\subseteq \{0,1\}^n$, let $\rem'$ denote the projection of $\rem$ onto the first $n-1$ coordinates. Also, let $\widehat \rem:= \widetilde \rem\setminus \rem$, where $\widetilde \rem$ is constructed from $\rem$ by flipping the last coordinate of each of its elements. The result below is key in giving a recursive construction of $\forb([0,1]^n,\rem)$.

\begin{proposition}\label{recursion}
 $\{0,1\}^n\setminus \rem=\left[\left(\{0,1\}^{n-1}\setminus \rem'\right)\times\{0,1\}\right]\cup \widehat \rem$.
\end{proposition}
\begin{proof}
Given $v\in\{0,1\}^n$, let $v'\in\{0,1\}^{n-1}$ and $\widetilde v\in\{0,1\}^n$ be the vectors obtained from $v$ by removing and by flipping its last coordinate, respectively.

Let $v\in \{0,1\}^n\setminus \rem$. If $\widetilde v\in \rem$, since $v\notin \rem$, we have $v\in \widehat \rem$. Otherwise $v'\notin \rem'$, and thus $v\in(\{0,1\}^{n-1}\setminus \rem')\times\{0,1\}$.

For the converse, note that $\widehat \rem\subseteq \{0,1\}^n\setminus \rem$. Finally, if $v\in (\{0,1\}^{n-1}\setminus \rem')\times\{0,1\}$, then $v'\notin \rem'$ and thus $v\notin \rem$.
\end{proof}

The second proof of Theorem~\ref{P01minusV} follows from Proposition~\ref{recursion} by induction. Suppose that $\forb([0,1]^{n-1},\rem')$ has an extended formulation with at most $(n-1)(|\rem'|+4)$ inequalities, which holds for $n=2$. Then we can describe $\forb([0,1]^{n-1},\rem')\times\{0,1\}$ using at most $(n-1)(|\rem'|+4) + 2$ inequalities. Since the polytope $\conv(\widehat \rem)$ requires at most $|\widehat \rem|$ inequalities \gustavo{in an extended formulation}, we obtain an extended formulation for $\forb([0,1]^n,\rem)$ of size no more than $[(n-1)(|\rem'|+4)+2+1]+[|\widehat \rem|+1]\leq n(|\rem|+4)$.

%
%
%

\subsection{General 0-1 polytopes}

\gustavo{In this subsection we analyze the general 0-1 case. We show that the encoding of $\rem$ plays an important role in the complexity of the problem.}

\subsubsection{\gustavo{Explicit $\rem$}}

\gustavo{In order to prove tractability of the forbidden vertices problem corresponding to general 0-1 tractable polytopes, we introduce the notion of $X$-separating faces for the 0-1 cube.}

\begin{definition}
Given $\rem\subseteq\{0,1\}^n$, we say that $\mathcal F\subseteq\FF([0,1]^n)$ is $\rem$-separating if $\{0,1\}^n\setminus \rem=\cup_{F\in\mathcal F}F\cap\{0,1\}^n$. We denote by $\mu(\rem)$ the minimal cardinality of an $\rem$-separating set.
\end{definition}

Clearly, if $\mathcal F$ is $\rem$-separating, then
$$\min\left\{c^\top x|\ x\in\{0,1\}^n\setminus \rem\right\}=\min_{F\in\mathcal F}\min\left\{c^\top x|\ x\in F\cap\{0,1\}^n\right\}.$$

Thus, if we can find an $\rem$-separating family of cardinality bounded by a polynomial on $n$ and $|\rem|$, then we can optimize in polynomial time over $\{0,1\}^n\setminus \rem$ by solving the inner minimization problem for each $F\in\mathcal F$ and then picking the smallest value. 

\begin{proposition}\label{muV}
For every nonempty set $\rem\subseteq \{0,1\}^n$, we have $\mu(\rem)\leq n|\rem|$.
\end{proposition}
\begin{proof}
 \gustavo{For each $y\in\{0,1\}^n\setminus\rem$, let $0\leq k\leq n-1$ be the size of the longest common prefix between $y$ and any element of $\rem$, and consider the face $F=F(y):=\{x\in [0,1]^n|\ x_i=y_i\ \forall 1\leq i\leq k+1\}=(y_1,\ldots,y_k,y_{k+1})\times[0,1]^{n-k-1}$. Then the collection $\mathcal F:=\{F(y)|\ y\in\{0,1\}^n\setminus\rem\}$ is $\rem$-separating since any $y\in\{0,1\}^n\setminus\rem$ belongs to $F(y)$ and no element of $\rem$ lies in any $F(y)$ by maximality of $k$. Clearly, $|\mathcal F|\leq n|\rem|$ since each face in $\mathcal F$ is of the form $(v_1,\ldots,v_k,1-v_{k+1})\times[0,1]^{n-k-1}$ for some $v\in\rem$.}
\end{proof}

\gustavo{In other words, letting} $\rem^i$ be the projection of $\rem$ onto the first $i$ components and $\widehat \rem^i:=(\rem^{i-1}\times\{0,1\})\setminus \rem^i$, where $\widehat \rem^1:=\{0,1\}\setminus \rem^1$, we have
$$\{0,1\}^n\setminus \rem=\bigcup_{i=1}^n\left[\widehat \rem^i\times\{0,1\}^{n-i}\right].$$

\gustavo{Moreover, it also follows from the proof of Proposition~\ref{muV} that $\mu(\rem)$ is at most the number of neighbors of $\rem$ since if $(v_1,\ldots,v_k,1-v_{k+1},v_{k+2},\ldots,v_n)$ is a neighbor of $v\in\rem$ that also lies in $\rem$, then the face $\left\{(v_1,\ldots,v_k,1-v_{k+1})\right\}\times[0,1]^{n-k-1}$ in not included in $\mathcal F$ in the construction above.}

\gustavo{Now, let} $P\subseteq\RR^n$ be an arbitrary 0-1 polytope. Note that $\VV(P)\setminus \rem=\VV(P)\cap(\{0,1\}^n\setminus \rem)$. On the other hand, if $\mathcal F\subseteq\FF([0,1]^n)$ is $\rem$-separating, then $\{0,1\}^n\setminus \rem=\cup_{F\in\mathcal F}F\cap\{0,1\}^n$. Combining these two expressions, we get
$$\VV(P)\setminus \rem=\bigcup_{F\in\mathcal F}\VV(P)\cap F\cap\{0,1\}^n=\bigcup_{F\in\mathcal F}P\cap F\cap\{0,1\}^n.$$

Note that since $P$ has 0-1 vertices and $F$ is a face of the unit \gustavo{cube}, then $P\cap F$ is a 0-1 polytope. Moreover, if $P$ is tractable, so is $P\cap F$. Recalling that $\mu(\rem)\leq n|\rem|$ from Proposition~\ref{muV}, we obtain

\begin{theorem}\label{poly01}
 If $P\subseteq\RR^n$ is a tractable 0-1 polytope, then the forbidden-vertices problem is polynomially solvable.
\end{theorem}

In fact, a compact extended formulation for $\VV(P)\setminus \rem$ is available when $P$ has one.

\begin{proposition}\label{bound1}
 For every 0-1 polytope $P$ and for every nonempty set $\rem\subseteq\VV(P)$, we have
$$\xc(\forb(P,\rem))\leq \mu(\rem)(\xc(P)+1).$$
\end{proposition}
\begin{proof}
 The result follows from
$$\forb(P,\rem)=\conv\left(\bigcup_{F\in\mathcal F}P\cap F\cap\{0,1\}^n\right)=\conv\left(\bigcup_{F\in\mathcal F}F\right),$$

Lemma~\ref{disj}, and $\xc(F)\leq\xc(P)$ for any face $F$ of $P$.
\end{proof}

Observe that when $P$ is tractable but its facet description is not provided, Theorem~\ref{poly01} is in contrast to Theorem~\ref{cutpoly}. Having all vertices with at most two possible values for each component is crucial to retain tractability when $\rem$ is given as a list. However, when $\rem$ is given by a face of $P$, the forbidden-vertices problem can become intractable even in the 0-1 case.

\subsubsection{\gustavo{Implicit $\rem$}}

Let $\tsp(n)$ denote the convex hull of the characteristic vectors of Hamiltonian cycles in the complete graph $K_n$. Also, let $\sub(n)$ denote the subtour-elimination polytope for $K_n$ with edge set $E_n$.

\begin{theorem}\label{binary_facet}
 For each $n$, there exists a 0-1 polytope $P_n\subseteq\RR^{n(n-1)/2}$ and a facet $F_n\in\FT(P_n)$ such that linear optimization over $P_n$ can be done in polynomial time and $\xc(P_n)$ is polynomially bounded, but linear optimization over $\VV(P_n)\setminus\VV(F_n)$ is $\mathcal{NP}$-hard and $\xc(\forb(P_n,\VV(F_n)))$ grows \gustavo{exponentially}.
\end{theorem}
\begin{proof}
Given a positive integer $n$, consider $T^+_n:=\{x\in\{0,1\}^{E_n}|\ \sum_{e\in E_n}x_e=n+1\}$, $T^-_n:=\{x\in\{0,1\}^{E_n}|\ \sum_{e\in E_n}x_e=n-1\}$, and $H_n:=\tsp(n)\cap\{0,1\}^{E_n}$. The idea is to ``sandwich'' $H_n$ between $T^-_n$ and $T^+_n$ to obtain tractability, and then remove $T^-_n$ to obtain hardness.

We first show that linear optimization over $T_n^-\cup H_n\cup T_n^+$ is polynomially solvable. Given $c\in\RR^{n(n-1)/2}$, consider $\max\{c^\top x|\ x\in T_n^-\cup H_n\cup T_n^+\}$. Let $x^-$ and $x^+$ be the best solution in $T_n^-$ and $T_n^+$, respectively, and note that $x^-$ and $x^+$ are trivial to find. Let $m$ be the number of nonnegative components of $c$. If $m\geq n+1$, then $x^+$ is optimal. If $m\leq n-1$, then $x^-$ is optimal. If $m=n$, let $x^n\in\{0,1\}^{E_n}$ have a 1 at position $e$ if and only if $c_e\geq 0$. If $x^n$ belongs to $H_n$, which is easy to verify, then it is optimal. Otherwise either $x^-$ or $x^+$ is an optimal solution.

Now we show that linear optimization over $H_n\cup T_n^+$ is $\mathcal{NP}$-hard. Given $c\in\RR^{n(n-1)/2}$ with $c>0$, consider $\min\{c^\top x|\ x\in H_n\}$. Let $\bar c:=\max\{c_e|\ e\in E_n\}$ and define $d_e:=c_e+n\bar c$. Consider $\min\{d^\top x|\ x\in H_n\cup T_n^+\}$. For any $x\in T_n^+$, we have $d^\top x=(n+1)n\bar c + c^\top x> (n+1)n\bar c$. For any $x\in H_n$, we have $d^\top x=n^2\bar c + c^\top x\leq n^2\bar c+n\bar c=(n+1)n\bar c$. Hence, the optimal solution to the latter problem belongs to $H_n$ and defines a tour of minimal length with respect to $c$.

Letting $P_n:=\conv(T_n^-\cup H_n\cup T_n^+)$, we have that $P_n$ is a tractable 0-1 polytope, $\sum_{e\in E_n}x_e\geq n-1$ defines a facet $F_n$ of $P_n$, and $\VV(P_n)\setminus\VV(F_n)=H_n\cup T_n^+$, which is an intractable set. Now, since $\forb(P_n,\VV(F_n))=\conv(H_n\cup T_n^+)$, we have that $\sum_{e\in E_n}x_e\geq n$ defines a facet of $\forb(P_n,\VV(F_n))$ and $\forb(P_n,\VV(F_n))\cap\{x\in\RR^{n(n-1)/2}|\ \sum_{e\in E_n}x_e=n\}=\tsp(n)$. Therefore, $\xc(\forb(P_n,\VV(F_n)))$ is \gustavo{exponential} in $n$ \cite{rothvoss2013matching}. It remains to show that $\xc(P_n)$ is polynomial in $n$.

Let $T_n:=\{x\in\{0,1\}^{E_n}|\ \sum_{e\in E_n}x_e=n\}$ and let $\overline H_n:=T_n\setminus H_n$ be the set of incidence vectors of $n$-subsets of $E_n$ that do not define a Hamiltonian cycle. Given $x\in\{0,1\}^{E_n}$, let $N(x)$ be the set of neighbors of $x$ in $[0,1]^{E_n}$, let $L(x)$ be the half-space spanned by $N(x)$ that does not contain $x$, and let $C(x):=[0,1]^{E_n}\setminus L(x)$. Finally, let $\Delta_n:=\conv(T^-_n\cup T_n\cup T^+_n)\gustavo{=\{x\in[0,1]^{E_n}|\ n-1\leq \sum_{e\in E_n}x_e\leq n+1\}}$.

We claim that $P_n=\conv(T^-_n\cup \sub(n) \cup T^+_n)$. By definition, we have $P_n\subseteq\conv(T^-_n\cup \sub(n) \cup T^-_n)$. To show the reverse inclusion, it suffices to show $\sub(n)\subseteq P_n$. \gustavo{Note that any two distinct elements in $T_n$ can have at most $|E_n|-2$ tight inequalities in common from those defining $\Delta_n$. Thus, $T_n$ defines an independent set in the graph of $\Delta_n$. Moreover, for each $x\in T_n$ the set of neighbors in $\Delta_n$ is $N(x)$ and thus all vertices in $T_n$ are simple. As $\overline H_n\subseteq T_n$, we have that $\overline H_n$ is simple and independent,} and by Corollary~\ref{stable} we have
$$P_n=\Delta_n\cap\bigcap_{x\in\overline H_n}L(x)=\Delta_n\setminus\bigcup_{x\in\overline H_n}C(x).$$

Since $\sub(n)\subseteq\Delta_n$, from the second equation above, it suffices to show $C(x)\cap\sub(n)=\emptyset$ for all $x\in\overline H_n$. For this, note that for any $x\in\overline H_n$, there exists a set $\emptyset\neq S\subsetneq V_n$ such that $x(\delta(S))\leq 1$, which implies $y(\delta(S))\leq 2$ for all $y\in N(x)$. Thus $C(x)\cap\sub(n)=\emptyset$ as $x(\delta(S))\geq 2$ is valid for $\sub(n)$.

Finally, applying disjunctive programming and since $\xc(\sub(n))$ is polynomial in $n$ \cite{yannakakis1991expressing}, we conclude that $P_n$ has an extended formulation of polynomial size.
\end{proof}

To conclude this section, consider the case where $P$ is explicitly given and $\rem$ is given as a facet of $P$. Although we are unable to establish the complexity of the forbidden-vertices problem in this setting, we present a tractable case and discuss an extension.

\begin{proposition}\label{faceTU}
Let $P=\{x\in\RR^n|\ Ax\leq b\}$ be a 0-1 polytope, where $A$ is TU and $b$ is integral. Let $F$ be the face of $P$ defined by $a_i^\top x = b_i$. Then
$$\forb(P,\VV(F))=P\cap\{x\in\RR^n|\ a_i^\top x\leq b_i-1\}.$$
\end{proposition}
\begin{proof}
We have 
$$\VV(P)\setminus\VV(F) = P\cap\{x\in\{0,1\}^n|\ a_i^\top x\leq b_i-1\}.$$

Since $A$ is TU and $b$ in integral, the set $P\cap\{x\in\RR^n|\ a_i^\top x\leq b_i-1\}$ is an integral polyhedron contained in $P$, which is a 0-1 polytope.
\end{proof}

Since any face is the intersection of a subset of facets, the above result implies that removing a single face can be efficiently done by disjunctive programming in the context of Proposition~\ref{faceTU}. Also, if we want to remove a list of facets, that is, $\rem=\cup_{F\in\mathcal F}\VV(F)$ and $\mathcal F$ is a subset of the facets of $P$, then we can solve the problem by removing one facet at a time. However, if $\mathcal F$ is a list of faces, then the problem becomes hard in general.

\begin{proposition}
If $\mathcal F$ is a list of faces of $[0,1]^n$,  then optimizing a linear function over $\{0,1\}^n\setminus \cup_{F\in\mathcal F}\VV(F)$ is $\mathcal{NP}$-hard.
\end{proposition}
\begin{proof}
Let $G=(V,E)$ be a graph. Consider the problem of finding a minimum cardinality vertex cover of $G$, which can be formulated as
\begin{eqnarray*}
\min & \sum_{i\in V}x_i\\
\st &x_i+x_j\geq 1 &\forall \{i,j\}\in E\\
& x_i\in\{0,1\} &\forall i\in V.
\end{eqnarray*}

Construct $\mathcal F$ by adding a face of the form $F=\{x\in[0,1]^n|\ x_i=0,\ x_j=0\}$ for each $\{i,j\}\in E$. Then the vertex cover problem, which is $\mathcal{NP}$-hard, reduces to optimization of a linear function over $\{0,1\}^n\setminus \cup_{F\in\mathcal F}\VV(F)$.
\end{proof}

\section{Applications}\label{app}

\subsection{$k$-best solutions}

The $k$-best problem defined below is closely related to removing vertices.

\begin{definition}\label{def_kbest}
 Given a nonempty 0-1 polytope $P\subseteq\RR^n$, a vector $c\in\RR^n$, and a positive integer $k$, the $k$-best problem is to either assert $|\VV(P)|\leq k$ and return $\VV(P)$, or to return $v_1,\ldots,v_k\in\VV(P)$, all distinct, such that  $\max\{c^\top v_i|\ i=1,\ldots,k\}\leq \min\{c^\top v|\ v\in\VV(P)\setminus\{v_1,\ldots,v_k\}\}$.
\end{definition}

Since we can sequentially remove vertices from 0-1 polytopes, we can prove the following.

\begin{proposition}
Let $P\subseteq[0,1]^n$ be a tractable 0-1 polytope. Then, for any $c\in\RR^n$, the $k$-best problem can be solved in polynomial time on $k$ and $n$.
\end{proposition}
\begin{proof}
For each $i=1,\ldots,k$, solve the problem 
\begin{eqnarray*}
(\mathcal P_i)\ \min & c^\top x\\
\st& x\in P_i,
\end{eqnarray*}

where $P_1:=P$, $P_i:=\forb(P_{i-1},\{v_{i-1}\})=\forb(P,\{v_1,\ldots,v_{i-1}\})$ for $i=2,\ldots,k$, and $v_i\in\VV(P_i)$ is an optimal solution to $(\mathcal P_i)$, if one exists, for $i=1,\ldots,k$. From Theorem~\ref{poly01}, we can solve each of these problems in polynomial time. In particular, if $(\mathcal P_i)$ is infeasible, we return $v_1,\ldots,v_{i-1}$. Otherwise, by construction, $v_1,\ldots,v_k$ satisfy the required properties. Clearly, the construction is done in polynomial time.
\end{proof}

The above complexity result was originally obtained in \cite{Lawler} building on ideas from \cite{Murty} by applying a branch-and-fix scheme.

\subsection{Binary all-different polytopes}

With edge-coloring of graphs in mind, the binary all-different polytope has been introduced in \cite{Lee}. It was furthermore studied in \cite{lee2007binary} and \cite{lee2005separating}. We consider a more general setting.

\begin{definition}
 Given a positive integer $k$, nonempty 0-1 polytopes $P_1,\ldots,P_k$ in $\RR^n$, and vectors $c_1,\ldots,c_k\in\RR^n$, the binary all-different problem is to solve
\begin{eqnarray*}
(\mathcal P)\ \min & \sum_{i=1}^k c_i^\top x_i\\
\st& x_i\in \VV(P_i) & i=1,\ldots,k\\
&x_i\neq x_j &1\leq i<j\leq k.
\end{eqnarray*}
\end{definition}

In \cite{Lee}, it was asked whether the above problem is polynomially solvable in the case $P_i=[0,1]^n$ for all $i=1,\ldots,k$. Using the tractability of the $k$-best problem, we give a positive answer even for the general case of distinct polytopes.

Given a graph $G=(V,E)$ and $U\subseteq V$, a $U$-matching in $G$ is a matching $M\subseteq E$ such that each vertex in $U$ is contained in some element of $M$.

\begin{theorem}\label{alldiffP}
If $P_i\subseteq\RR^n$ is a tractable nonempty 0-1 polytope for $i=1,\ldots,k$, then the binary all-different problem is polynomially solvable.
\end{theorem}
\begin{proof}
For each $i=1,\ldots,k$, let $S_i$ be the solution set of the $k$-best problem \gustavo{(Def.~\ref{def_kbest})} for $P_i$ and $c_i$. \gustavo{Observe that $|S_i|\leq k$.} Now, consider the bipartite graph $G=(S\cup R,E)$, where $S:=\cup_{i=1}^k S_i$ and $R:=\{1,\ldots,k\}$. For each $v\in S$ and $i\in R$, we include the arc $\{v,i\}$ in $E$ if and only if $v\in S_i$. Finally, for each $\{v,i\}\in E$, we set $w_{vi}:=c_i^\top v$.

We claim that $(\mathcal P)$ reduces to finding an $R$-matching in $G$ of minimum weight with respect to $w$. It is straightforward to verify that an $R$-matching in $G$ defines a feasible solution to $(\mathcal P)$ of equal value. Thus, it is enough to show that if $(\mathcal P)$ is feasible, then there exists an $R$-matching with the same optimal value. Indeed, let $(x_1,\ldots,x_k)$ be an optimal solution to $(\mathcal P)$ that does not define an $R$-matching, that is, such that $x_i\notin S_i$ for some $i=1,\ldots,k$. Then, we must have $|\VV(P_i)|>k$ and $|S_i|=k$. This latter condition and $x_i\notin S_i$ imply the existence of $v\in S_i$ such that $v\neq x_j$ for all \gustavo{$j=1,\ldots,k$}. Furthermore, by the definition of $S_i$, we also have $c_i^\top v\leq c_i^\top x_i$. Therefore, the vector $(x_1,\ldots,x_{i-1},v,x_{i+1},\ldots,x_k)$ is an optimal solution to $(\mathcal P)$ having its $i$-th \gustavo{subvector} in $S_i$. Iteratively applying the above reasoning to all components, we obtain an optimal solution to $(\mathcal P)$ given by an $R$-matching as desired.
\end{proof}

\section{Extension to integral polytopes}\label{integral}

In this section, we generalize the forbidden-vertices problem to integral polytopes, that is, to polytopes having integral extreme points, even allowing the removal of points that are not vertices. We show that for an important class of integral polytopes the resulting problem is tractable.

For an integral polytope $P\subseteq\RR^n$ and $\rem\subseteq P\cap\ZZ^n$, we define $\forb_I(P,\rem):=\conv((P\cap\ZZ^n)\setminus \rem)$.

\begin{definition}
 Given an integral polytope $P\subseteq\RR^n$, a set $\rem\subseteq P\cap\ZZ^n$ of integral vectors, and a vector $c\in\RR^n$, the forbidden-vectors problem asks to either assert $(P\cap\ZZ^n)\setminus \rem=\emptyset$, or to return a minimizer of $c^\top x$ over $(P\cap\ZZ^n)\setminus \rem$ otherwise.
\end{definition}

Given vectors $l,u\in\RR^n$ with $l\leq u$, we denote $[l,u]:=\{x\in\RR^n|\ l_i\leq x_i\leq u_i, i=1,\ldots,n\}$. We term these sets as boxes.

\begin{definition}
 An integral polytope $P\subseteq\RR^n$ is box-integral if for any pair of vectors $l,u\in\ZZ^n$ with $l\leq u$, the polytope $P\cap[l,u]$ is integral.
\end{definition}

Polytopes defined by a TU matrix and an integral right-hand-side, or by a box-TDI system, are examples of box-integral polytopes. Further note that if $P$ is tractable and box-integral, so is $P\cap[l,u]$. When both conditions are met, we say that $P$ is box-tractable.

With arguments analogous to that of the 0-1 case, we can verify the following result.

\begin{theorem}\label{vectors}
 If $P\subseteq\RR^n$ is a box-tractable polytope, then, given a list $\rem\subseteq P\cap\ZZ^n$, the forbidden-vectors problem is polynomially solvable. Moreover,
$$\xc(\forb_I(P,\rem))\leq 2n|\rem|(\xc(P)+1).$$
\end{theorem}

\begin{proof}
Since $P$ is bounded, it is contained in a box. Without lost of generality and to simplify the exposition, we may assume that $P\subseteq [0,r-1]^n$ for some $r\geq 2$. As in the 0-1 case, we first address the case $P=[0,r-1]^n$, for which we provide two extended formulations for $\forb_I(P,\rem)$ involving $\mathcal O(n|\rem|)$ variables and constraints.

The first extended formulation relies on the mapping $\phi(x):=\sum_{i=1}^n r^{i-1}x_i$ for $x\in [0,r-1]^n$, which defines a bijection with $\{0,\ldots,r^n-1\}$. Letting $K_r(a,b):=\{x\in\{0,\ldots,r-1\}^n|\ a\leq \phi(x)\leq b\}$, we have that $\forb_I(P,\rem)$ is the convex hull of the union of at most $|\rem|+1$ sets of the form $K_r(a,b)$. Since $\conv(K_r(a,b))$ has $\mathcal O(n)$ facets \cite{gupten}, by disjunctive programming we obtain an extended formulation for $\forb_I(P,\rem)$ having $\mathcal O(n|\rem|)$ inequalities.

For the second extended formulation, let $\rem'$ denote the projection of $\rem$ onto the first $n-1$ coordinates and set $\widehat \rem:= (\rem'\times\{0,\ldots,r-1\})\setminus \rem$. Along the lines of Proposition~\ref{recursion}, we have
$$\{0,\ldots,r-1\}^n\setminus \rem=\left[\left(\{0,\ldots,r-1\}^{n-1}\setminus \rem'\right)\times\{0,\ldots,r-1\}\right]\cup \widehat \rem.$$

Although $\widehat \rem$ can have up to $r|\rem|$ elements, we also see that $\widehat \rem$ is the union of at most $2|\rem|$ sets of the form $v\times\{\alpha,\ldots,\beta\}$ for $v\in \rem'$ and integers $0\leq \alpha\leq \beta\leq r-1$. More precisely, for each $v\in \rem'$, there exist integers $0\leq \alpha^v_1\leq \beta^v_1 < \alpha^v_2 \leq \beta^v_2 < \cdots < \alpha^v_{q_v}\leq\beta^v_{q_v}\leq r-1$ such that
$$\widehat \rem=\bigcup_{v\in \rem'}\bigcup_{l=1}^{q_v}v\times\{\alpha^v_l,\ldots,\beta^v_l\}$$

and $\sum_{v\in\rem'}q_v\leq 2|\rem|$. Therefore, $\conv(\widehat \rem)$ can be described with $\mathcal O(|\rem|)$ inequalities. Then a recursive construction of an extended formulation for $\forb_I(P,\rem)$ is analogous to the binary case and involves $\mathcal O(n|\rem|)$ variables and constraints.

In order to address the general case, we first show how to cover $\{0,\ldots,r-1\}^n\setminus \rem$ with boxes. For each $i=1,\ldots,n$, let $\rem^i$ be the projection of $\rem$ onto the first $i$ components and let $\widehat \rem^i:=(\rem^{i-1}\times\{0,\ldots,r-1\})\setminus\rem^i$, where $\widehat \rem^1:=\{0,\ldots,r-1\}\setminus \rem^1$. Working the recursion backwards yields
$$\{0,\ldots,r-1\}^n\setminus \rem=\bigcup_{i=1}^n\left[\widehat \rem^i\times\{0,\ldots,r-1\}^{n-i}\right].$$

Combining the last two expressions, we arrive at
$$\{0,\ldots,r-1\}^n\setminus\rem=\bigcup_{i=1}^n\bigcup_{v\in \rem^{i-1}}\bigcup_{l=1}^{q_v}v\times\{\alpha^v_l,\ldots,\beta^v_l\}\times\{0,\ldots,r-1\}^{n-i}.$$

The right-hand-side defines a family $\mathcal B$ of at most $2n|\rem|$ boxes in $\RR^n$, yielding
$$\{0,\ldots,r-1\}^n\setminus\rem=\bigcup_{[l,u]\in\mathcal B}[l,u]\cap\ZZ^n.$$

Finally, if $P\subseteq[0,r-1]^n$, then
$$(P\cap\ZZ^n)\setminus \rem= (P\cap\ZZ^n)\cap(\{0,\ldots,r-1\}^n\setminus\rem) =\bigcup_{[l,u]\in\mathcal B}P\cap[l,u]\cap\ZZ^n.$$

Moreover, if $P$ is box-tractable, then
$$\forb_I(P,\rem)=\conv\left(\bigcup_{[l,u]\in\mathcal B}\conv\left(P\cap[l,u]\cap\ZZ^n\right)\right)=\conv\left(\bigcup_{[l,u]\in\mathcal B}P\cap[l,u]\right),$$

where each term within the union is a tractable set.
\end{proof}

The $k$-best problem and the binary all-different problem can be extended to the case of integral vectors as follows.

\begin{definition}
 Given a nonempty integral polytope $P\subseteq\RR^n$, a vector $c\in\RR^n$, and a positive integer $k$, the integral $k$-best problem is to either assert $|P\cap\ZZ^n|\leq k$ and return $P\cap\ZZ^n$, or to return $v_1,\ldots,v_k\in P\cap\ZZ^n$, all distinct, such that  $\max\{c^\top v_i|\ i=1,\ldots,k\}\leq \min\{c^\top v|\ v\in\ (P\cap\ZZ^n)\setminus\{v_1,\ldots,v_k\}\}$.
\end{definition}

\begin{definition}
 Given a positive integer $k$, nonempty integral polytopes $P_1,\ldots,P_k$ in $\RR^n$, and vectors $c_1,\ldots,c_k\in\RR^n$, the integral all-different problem is to solve
\begin{eqnarray*}
(\mathcal P)\ \min & \sum_{i=1}^k c_i^\top x_i\\
\st& x_i\in P\cap\ZZ^n & i=1,\ldots,k\\
&x_i\neq x_j &1\leq i<j\leq k.
\end{eqnarray*}
\end{definition}

The above problems can be shown to be polynomially solvable if the underlying polytopes are box-tractable.

\textbf{Acknowledgments}

We thank Marc Pfetsch for pointing out the example in Proposition~\ref{single}. 
This research has been supported in part by the Air Force Office of Scientific Research (Grant \#FA9550-12-1-0154) and by Deutsche Forschungsgemeinschaft (KA 1616/4-1).

\bibliographystyle{amsplain} 
\bibliography{forbidden_vertices_preprint_rev1}

\end{document}